\documentclass[reqno]{amsart}

\usepackage{amssymb}
\usepackage{amscd}
\usepackage[all]{xy}

\theoremstyle{plain}
\newtheorem{thm}{Theorem}[section]
\newtheorem{prop}[thm]{Proposition}

\newtheorem{cor}[thm]{Corollary}
\newtheorem*{thm*}{Theorem}
\theoremstyle{definition}
\newtheorem{defn}[thm]{Definition}

\theoremstyle{remark}

\newtheorem{remark}[thm]{Remark}

\newcommand{\F}{\mathbb F}
\newcommand{\Z}{\mathbb Z}

\newcommand{\ms}{\medskip\noindent}

\newcommand\vv{\,\vert\,}
\newcommand\sub{\subseteq}
\newcommand\rar[1]{\stackrel{#1}{\to}}

\newcommand\Hom[3]{\operatorname{Hom}_{#1}(#2,#3)}
\newcommand\ext[4]{\operatorname{Ext}_{#1}^{#2}({#3},{#4})}
\newcommand\dpth{\text{depth}\,}
\newcommand{\dm}{\text{dim}\,}
\newcommand{\fg}{finitely generated }
\newcommand{\an}{\text{ann}\,}
\newcommand{\pd}{\text{pd}\,}
\newcommand{\as}{\text{Ass}\,}
\newcommand{\cod}{\text{codim}\,}

\title{On Carlson's depth conjecture } 
\author{James A. Schafer}

\date{\today}

\address{\begin{flushleft}\quad Department of Mathematics \\
\quad University of Maryland \\
\quad College Park, Maryland 20742
\end{flushleft}}

\email{jas@math.umd.edu}

\begin{document} 

\maketitle

\section{Introduction}

Let $A=\oplus_{i\geq 0}A_i$ be a \fg graded commutative $k$-algebra where $k$ is a field and $M=\oplus_{i\geq 0}M_i$ a \fg graded $k$-module.

\begin{defn} $\omega_A(M)=\min\{\dm P\vv P\in\as_AM\}$.  For $M=A$ denote this simply by $\omega\,A$. \end{defn}

It is well known (shown originally by Serre) that $\dpth_AM\leq \omega_AM$.  For a proof of this as well as other standard results of commutative algebra adapted to the graded commutative $k$-algebra setting a good reference is \cite{P}. Clearly $\omega_AM\leq \dim_AM$ and so $\dpth_AM\leq \omega_AM\leq \dim_AM$ for all \fg $M$.

It is quit easy to produce examples where $\dpth_AM<\omega_AM$.  However based on the work of Benson and Carlson in \cite{B}, Carlson in \cite{C} asked whether it might not be the case that $\omega\,A=\dpth\,A$ if $A=H^*(G,k)$ where $k$ is a field whose characteristic divides the order of the finite group $G$.  This is obviously true if $H^*(G,k)$ is Cohen-Macaulay since $\dpth A=\dim A$ in that case. Carlson showed in the same paper the conjecture is true if $\dm H^*(G,k)=2$. Except for innumerable calculations showing the conjecture was true the only other general result is one of D.J Green in \cite{G}.  A basic result of Duflot \cite{D} states that if $k$ is a field of characteristic $p>0$ then the $p$-rank of the center of $G$ is a lower bound for the depth of $H^*(G,k)$.  Green shows that if $G$ is a $p$-group and the depth of $H^*(G,k)$ equals Duflot lower bound then $\dpth\,H^*(G,k)=\omega\,H^*(G,k)$.

Note that because of the result of Duflot if $\dim H^*(G,k)=2$ the gap between the dimension and the depth of $H^*(G,k)$ is at most one.  Carlson's result could be interpreted to say that if the gap is less than or equal to one (and the dimension is two) then the conjecture is true.  It is the aim of this paper to demonstrate the following

\begin{thm*} Let $G$ be a finite group and $k$ a field whose characteristic divides the order of $G$. If $\dim H^*(G,k)-\dpth H^*(G,k)=1$, then $\dpth\,H^*(G,k)=\omega\,H^*(G,k)$.
\end{thm*}

\section{Reduction to the case $A=k[x_1,\dots,x_d]$}

\begin{prop}\label{red} 1) Let $f:S\to A$ be a homomorphism of Noetherian rings and suppose $A$ is \fg as an $S$-module. Let $M$ be a \fg $A$-module.  Then $\as_SM=f^{-1}(\as_AM)$. If $P\in\as_AM$ then $\dm P=\dm f^{-1}P$ and hence $\omega_AM=\omega_SM$.
\par 2) Let $f:S\to A$ be an epimorphism of Noetherian rings and $M$ a \fg $A$-module.  Suppose either $S$ and $A$ are local and $f(\frak m_S)\sub \frak m_A$ or $S$ and $A$ are graded commutative \fg $k$ algebras where $k$ is a field, $f$ is a graded homomorphism and $M$ a \fg graded module. Then $\dpth_SM=\dpth_AM$.
\par 3) In the graded commutative case $f$ can also be an inclusion with $A$ integral over $S$.\end{prop}
\begin{proof} 1) If $Q$ is an $A$-primary in $M$ with associated prime $P=r(\an_AM/Q)$ then it is easily seen that $Q$ is $S$-primary with associated prime $f^{-1}P=r(\an_SM/Q)$.  Hence if
$0=\bigcap Q$ is a minimal primary decomposition of $0$ as an $A$-module we obtain a minimal decomposition of $0$ as an $S$-module by intesecting all those $Q$ whose associated primes have the same inverse image under $f$.  If $P_i$ is the prime associated to the $A$-primaray module $Q_i$ then this gives an $S$-primary decompostion $(0)=\cap\bar Q_{r_j}$ where $\bar Q_{r_j}=\cap\{Q_i\vv f^{-1}P_{r_i}=f^{-1}P_{r_j}\}$. This is a primary decomposition of $(0)$ with distinct primes and it must be minimal for if $\cap_{j'\ne j}\bar Q_{r_{j'}}\sub \bar Q_{r_j}$ then $\cap_{s\ne r_j}Q_s\sub Q_{r_j}$ which is impossible since $(0)=\cap Q_i$ is a minimal primary decomposition.  The last statement follows since $S/f^{-1}P\sub A/P$ and the latter is \fg over the former.
\par 2) If $f$ is an epimorphism and $\{x_1,\dots,x_k\}$ is a maximal $M$-sequence in $A^+$, we may choose $y_i\in S^+$ (of the same degree) with $f(y_i)=x_i$. Since $y_im=x_im$ for all $m\in M$, $\{y_1,\dots,y_k\}$ is an $M$-sequence in $S^+$.  Hence $\dpth_SM\geq\dpth_AM$. If $\{y_1,\dots y_t\}$ is an $M$-sequence in $S^+$ then since $y_im=f(y_i)m$, $\{f(y_1),\dots,f(y_t)\}$ is an $M$-sequence in $A^+$ and so $\dpth_AM\geq\dpth_SM$.
\par 3) In \cite{Ev} Evens shows this is true in the strictly commutative case if $A$ is integral extension of $S$.  An $M$-sequence for $M$ consists of algebraically independent elements since any $M$-sequence can be exended to a system of parameters for $M$ and these are algebraically independent, \cite{B1}, theorem 2.2.7. Hence any $M$-sequence is contained in $S^{ev}$ or $A^{ev}$.  Hence $\dpth_{S^{ev}}M=\dpth_SM$ and therefore $\dpth_AM=\dpth_SM$.\end{proof}



\section{Modules over regular local or $*$-local domains}

In this section, $S$ will always be a Noetherian, regular ring. By $(S,\frak m)$ we will mean a regular \fg graded commutative $k$-algebra over a field $k$ with $S_0=k$ and where $\frak m$ is the maximal homogeneous ideal of elements of positive degree or a regular local ring and $\frak m$ is the maximal ideal. We will just say $(S,\frak m)$ is a local ring in both cases.  If $M$ is a \fg $S$-module, graded in the graded case, $\pd M$ denotes the projective dimension of $M$ over $S$ and $H^i_{J}(M)$ is the local cohomology of the \fg module $M$ with respect to the ideal $J\sub S$.

\begin{thm} i) For \fg $M$, $\min\{j\vv \ext SjMS)\neq 0\}=\cod M$.
\par ii) If $(S,\frak m)$ is local then $\max\{j\vv \ext SjMS\neq 0\}= \pd M$.
\end{thm}
\begin{proof} i) $\cod M=\cod(\an M)$, $\cod I=\dpth I=\dpth(I,S)$ for all ideals $I\sub S$ since $S$ is Cohen-Macauley. But $\dpth(\an M,S)\geq n$ if and only if $\ext SjMS=0$ for all $j<n$ by \cite{E}, Proposition 18.4.
\par The proof of ii) uses the following duality theorem for Gorenstein local rings. Proofs for the local case may be found in \cite{H}, 11.8 and for the graded $k$-algebra case in \cite{BS} 11.2.5 and \cite{BH}, 3.58.
$N^\vee$ is the Matlis dual $\Hom SN{E}$ where $E=E(S/\frak m)$ is the injective envelope of $S/\frak m=k$.  In the graded case $N^\vee$ may be identified, \cite{BH}, 3.6.16, with graded Hom functor $$*\Hom kNk\text{  where } *\Hom kNk_s=\Hom k{N_{-s}}k.$$

\begin{remark} In the graded case, define graded $\operatorname{Hom}$, $*\operatorname{Hom}_S(M,N)_t=\{f:M\to N\vv f(M_u)\sub N_{u+t} \text{ for all }u\}$. If $M$ is \fg then $*\operatorname{Hom}_S(M,N)=\operatorname{Hom}_S(M,N)$ and therefore $*\ext S*MN=\ext S*MN$ for all $N$.\end{remark}

\begin{thm} \label{Hoch} Let $(S,\frak m)$ be a Gorenstein, local ring of dimension $d$ and $M$ a \fg $S$-module. Then there exists a natural isomorphism $$H^i_{\frak m}M\cong \ext S{d-i}MS^\vee.$$\end{thm}
\begin{remark} In the graded case \cite{BH} we have an isomorphism of graded modules $(H^i_{\frak m}M)^\vee\simeq \ext S{d-i}M{S[a]}$ for some $a$. Since $H^j_{\frak m}M$ is Artinian and duality is an anti equivalence of the full subcategories of the category of graded $(S,\frak m)$-modules consisting of the \fg modules and the Artinian modules we obtain  the dedsired isomorphism. If $S=k[x_1,\dots,x_n]$ then $a=-\sum\text{degree}(x_i)$.\cite{BH} p.140.\end{remark} 

\ms  Proof of ii) By Auslander-Buchsbaum $\text{pd}\,M+\dpth M=d$.  But 
$$\dpth\,M=\min\{j\vv H^j_{\frak m}M\neq 0\}=d-\max\{e\vv\ext SeMS^\vee\neq 0\}$$ and so
$$ \text{pd}M=\max\{e\vv\ext SeMS^\vee\neq 0\}.$$
But the Matlis duality functor is faithful (obvious in the graded case) and hence the conclusion.
\end{proof}

The following result by Eisenbud, Huenke and Vasconselos will be essential.
 
\begin{thm}[\cite{EV}, Theorem 1.1] Let $M$ be a \fg module over a regular domain $S$ and set $I_e=\an\ext SeMS$:
\begin{enumerate}\item $\cod I_e \geq e$ and $M/(0:_MI_e)$ has no associated primes of $\cod\,e$.  In particular, if $P\in\text{Spec}\,S$ and $\cod P=e$ then $P\in\as (M)$ iff $P\in V(I_e)$.
\item If $c=\cod M$ then $\text{hull}\,(0,M)$ is the kernel of the natural map
$$\varphi:M\to\ext Sc{\ext ScMS}S). $$\end{enumerate}\end{thm}  

\noindent If $M$ is a \fg $S$-module, let $(\as M)_e=\{P\in \as M\vv \cod P=e\}$.

\begin{prop} Let $(S,\frak m)$ be a regular domain and $M$ a \fg $S$-module. Then
$$ \ext SjMS=0 \text{ implies }(\as M)_j=\emptyset\text{ hence} $$
$$ \max\{j\vv (\as M)_j\neq \emptyset\}\leq\max\{ j\vv \ext SjMS\neq 0\}.$$\end{prop}
\begin{proof} Let $P\in\as M$ be of codimension $j$. Then $P_P$ is an $S_P$-associated prime of $M_P$ which is non-zero since $P\in V(\an M)=\text{Supp}\,M$. Hence $S_P/P_P\sub M_P$ and $H^0_P(M_P)\neq 0$ since $H^0_P(M_P)$ is the maximal submodule $M'$ of $M_P$ of finite type. $H^0_P(M_P/M')=0$ and therefore $\dpth M_P/M'>0$ and $\pd M_P/M'<j$ by Auslander-Buchsbaum.  Hence 
$$\operatorname{Ext}_{S_P}^j(M_P,S_P)\to \operatorname{Ext}_{S_P}^j(M',S_P)  $$
is an isomorphism. Since $\dm S_P=j$, $\operatorname{Ext}_{S_P}^j(M',S_P)\neq 0$ and therefore $$\operatorname{Ext}_{S_P}^j(M_P,S_P)=\ext SjMS_P\neq 0.$$
\end{proof}

\begin{thm} Let $e=\pd M$ and $I=I_e=\an\ext SeMS$.  Then $(\as M)_e\neq\emptyset$ if and only if $H^0_I(M)\neq 0$.\end{thm}
\begin{proof} Since $e=\pd M$, $(\as M)_f=\emptyset$ for $f>e$. From the above result, $\cod I\geq e$ and $(\operatorname{Spec} S)_e\cap V(I)=(\as M)_e$. Therefore $(\as M)_e=\emptyset$ implies (is equivalent to) all associated primes $P$ of $M$ have $I\not\sub P$ which in turn implies $I_P=S_P$ and hence $(H^0_I(M))_P=H^0_{I_P}(M_P)=(0)$ but this implies $H^0_I(M)=(0)$ by [E], Cor. 3.5. Conversely if $(\as M)_e\neq\emptyset$ then there exists $P\in\as M$ with $I\sub P$. But $P_P\in\as_{S_P}M_P$ and so $\dpth(S_P,M_P)=0$.  But this is equivalent to $H^0_{P_P}M_P\neq (0)$ and since $H^0_{P_P}M_P\sub H^0_{I_P}M_P$ we have $(H^0_IM)_P\neq (0)$ and hence $H^0_IM\neq (0)$.\end{proof}

\begin{remark} An even easier proof is to note that by \cite{E} corollary 3.13 and \cite{EV}, $\as H^0_IM=(\as M)_e$ .\end{remark}

\begin{thm} Let $(S,\frak m)$ be a local, regular domain and $M$ a \fg $S$-module with $\pd M=e$.  Let $I=I_e=\an \ext SeMS$.  Then the following are equivalent.
\begin{enumerate}\item $\dpth M=\omega M$.
\item $H^0_{I}(M)\neq 0$.
\item $ \operatorname{Hom}_S(\ext SeMS,M)\neq 0$.
\item $\operatorname{Tor}^S_e(M,M)\neq 0$.
\end{enumerate}\end{thm}
\begin{proof} Since $\pd M=e$, $\ext SfMS=(0)$ if $f>e$, hence $(\as M)_f=\emptyset$ for $f>e$.
\par $(1)\iff (2)$. If $\dm S=s$ then $\dpth M=s-\pd M=s-e$ by Auslander-Buchsbaum.  On the other hand $\omega M=s-\max\{j\vv (\as M)j\neq\emptyset\}$.  Therefore $\dpth M=\omega M$ if and only if $(\as M)_e\neq\emptyset$ if and only if $H^0_IM\neq (0)$. 
\par $(2)\iff(3)$ Let $\overline M=\ext SeMS$. From [E], Proposition 18.4, for all $N$ with $\an N+\an\overline M\neq S$ 
$$\dpth(I,N)=\min\{r\vv \ext Sr{\overline M}N\neq 0\}.$$
But $\an M\sub \an \overline M$ and so
$$\dpth(I,M)=\min\{r\vv \ext Sr{\overline M}M\neq 0\}.$$
But from \cite{H}, 6.9, $\dpth(I,M)=\min\{t\vv H^t_{I}M\neq 0\}$.  Therefore 
$$ H^0_{I}M\neq 0\iff \dpth(I,M)=0\iff \operatorname{Hom}(\ext SeMS,M)\neq 0.$$
\par $(3)\iff(4)$.
If $F_*\to M\to 0$ is a minimal free $S$-resolution of $M$ then 
$$   \operatorname{Hom}_S(F_{e-1},S)\to  \operatorname{Hom}_S(F_e,S)\to \ext SeMS\to 0  $$
is exact and hence
$$0\to  \operatorname{Hom}_S(\ext SeMS,M)\to  \operatorname{Hom}_S( \operatorname{Hom}_S(F_e,S),M))\to \operatorname{Hom}( \operatorname{Hom}_S(F_{e-1},S),M))$$
is exact.  But for any \fg projective module $P$, there is a natural isomorphism $P\otimes M\to \operatorname{Hom}_S( \operatorname{Hom}_S(P,S),M)$ given by $x\otimes m\mapsto g$ where $g(f)=f(x)m$.  Hence there exists an exact sequence
$$  0\to  \operatorname{Hom}_S(\ext SeMS,M)\to F_e\otimes M\to F_{e-1}\otimes M  $$ 
and therefore $ \operatorname{Hom}_S(\ext SeMS,M)\simeq \operatorname{Tor}_e^S(M,M)$.  Hence $H^0_IM\neq 0$ if and only if $\operatorname{Tor}_e^S(M,M)\neq 0$.
\end{proof}

\begin{remark} In the graded case since $*\operatorname{Hom}_S(A,B)[a]=*\operatorname{Hom}_S(A[-a],B)=*\operatorname{Hom}_S(A,B[a])$ for any $a\in\Z$, we have for any $a,b\in\Z$, $*\ext SjAB\neq 0$ if and only if $*\ext Sj{A[a]}{B[b]}\neq 0$ . \end{remark}

\begin{thm} Let $(S,\frak m)=k[x_1,\dots,x_n]$ be a graded polynomial ring over the field $k$ and $M$ a \fg graded $S$-module of depth $d$. Then $$\dpth_SM=\omega_SM\text{ if and only if }*\operatorname{Hom}_S(M^{\vee},H^d_{\frak m}M)\neq(0)$$
where $M^{\vee}$ is the Matlis dual of $M$, 
$$M^{\vee}=*\operatorname{Hom}_k(M,k)\simeq *\operatorname{Hom}_S(M,E_S(k)). $$
\end{thm}
\begin{proof} By Auslander-Buchsbaum $\pd M=e=n-d$ and we know $\dpth_SM=\omega_SM$ if and only if $\operatorname{Hom}_S(\ext SeMS,M)\neq 0$.  If $\sigma=\sum\text{degree } x_i$ then the above remarks say $\dpth_SM=\omega_SM$ if and only if $\operatorname{Hom}_S(\ext SeM{S[-\sigma]},M)\neq 0$.  Since $\ext SeM{S[-\sigma]}$ is \fg, $$\operatorname{Hom}_S(\ext SeM{S[-\sigma]},M)=*\operatorname{Hom}_S(\ext SeM{S[-\sigma]},M).$$  Since Matlis duality is an anti-equivalence of the full subcategories of graded $S$-modules consisting of \fg graded modules and Artinian modules $\dpth_SM=\omega_SM$ if and only if $*\operatorname{Hom}_S(M^\vee,(\ext SeM{S[-\sigma]})^\vee)\neq 0$.
By local duality, \ref{Hoch}, we have $(\ext SeM{S[-\sigma]},M)^\vee\simeq H^d_{\frak m}M$.
\end{proof}

\section{Finitely generated graded $k$-algebras and $H^*(G,k)$}

\begin{thm} Let $(A,\frak m)$ be a \fg graded commutative $k$-algebra and $M$ a \fg graded module of depth $d$. 
Let $M^*$ be the $k$-dual of $M$ with $A$-action given by $(af)m)=f(am)$.  Then $\dpth_AM=\omega_AM$ if and only if $*\operatorname{Hom}_A(M^*,H^d_{\frak m}M)\neq (0)$. \end{thm} 
\begin{proof} Let $\pi:(S,\frak m_S)=(k[x_1,\dots, x_n],(x_1,\dots,x_n))\to (A,\frak m)$ be an epimorphism.  By \ref{red} $\dpth_AM=\omega_AM$ if and only if $\dpth_SM=\omega_SM$ and this occurs if and only if $\operatorname{Hom}_S(M^{\vee},H^d_{\frak m_S}M)\neq(0)$.  Now by change of rings for local cohomolgy there is a canonical isomorphism $H^i_{\frak m_S}M\simeq H^i_{\frak m}M$ since $\pi(\frak m_S)=\frak m$.  If the $A$-module $H^i_{\frak m}M$ is given an $S$-module structure by means of $\pi$ then this isomorphism is an isomorphism of $S$-modules.  Since $M^\vee\simeq M^*$ and the $A$-module and $S$-module structures on $M^*$ correspond under $\pi$ we can conclude that $\dpth_AM=\omega_AM$ if and only if $\operatorname{Hom}_S(M^*,H^d_{\frak m}M)\neq (0)$ where both $M^*$ and $H^d_\frak{m}M$ are $A$-modules which are $S$-modules via $\pi$. Since $\pi$ is an epimorphism $\operatorname{Hom}_S(M^*,H^d_{\frak m}M)=\operatorname{Hom}_A(M^*,H^d_{\frak m}M)$.  \end{proof}

\begin{cor} Suppose $A=H^*(G,k)$ then $\dpth A=\omega A$ if and only if there exists a non-zero graded $H^*(G,k)$-map,
$\lambda:H_*(G,k)\to H^{\dpth A}_{\frak m}(H^*(G,k))$ where $H_*(G,k)$ is an $H^*(G,k)$-module via cap product.
\end{cor}
\begin{proof} $(A^*)_s=\operatorname{Hom}_k(A^{-s},k)\simeq H_{-s}(G,k)$ and it is well known that the action of $A$ on $A^*$ is the cap product.\end{proof}

\begin{thm} Let $G$ be a finite group and $k$ a field whose characteristic divides the order of $G$. If $\dim H^*(G,k)-\dpth H^*(G,k)=1$ then $\dpth H^*(G,k)=\omega H^*(G,k)$.
\end{thm}
\begin{proof} In \cite{Gr} and again from a more algebraic aspect in \cite{B2} and \cite{B} a convergent spectral sequence $\{E_r^{p,q}\vv p\geq 0,\,r\geq 2\}$ of $H^*(G,k)$-modules is developed with $E_2^{p,q}=H^{p,q}_{\frak m}H^*(G,k)$ where $H^{p,q}_{\frak m}H^*(G,k))$ is the degree $q$ part of $H^p_{\frak m}(H^*(G,k))$ and converging to $H_{q-p}(G,k)$.  That is there is an decreasing filtration $F^pH_*(G,k)$ with $E_{\infty}^{p,-q}\simeq F^p(H_{q-p}(G,k)/F^{p+1}H_{q-p}(G,k)$. Grothendieck's Vanishing Theorem says $E_2^{p,q}=(0)$ if $p<\dpth H^*(G,k)$ or $p> \dim H^*(G,k)$ and that both $E_2^{\dpth, *}$ and $E_2^{\dim,*}$ are non-zero.  There is a graded $H^*(G,k)$ edge homomophism $\nu: H_*(G,k)\to E_2^{\dpth,*}=H_{\frak m}^{\dpth,*}(H^*(G,k))$ whose image is $E_{\infty}^{\dpth,*}$.  Hence if $E_{\infty}^{\dpth,*}\neq (0)$ we have a non-zero graded $H^*(G,k)$-homomorphism $$\nu: H_*(G,k)\to H^*_{\frak m}H^*(G,k)$$ and hence $\dpth H_*(G,k)=\omega H_*(G,k)$.  If $\dim H_*(G,k)-\dpth H_*(G,k)=1$ this spectral sequence has only two non-zero columns and so $E_2=E_{\infty}$. Grothendieck's Theorem gives the result. \end{proof}

\end{document}